\documentclass[11pt, reqno]{article}

\usepackage{amsmath,amssymb,amsfonts,amsthm}
\usepackage{color}

\usepackage[colorlinks=true,linkcolor=blue,citecolor=blue,urlcolor=blue,pdfborder={0 0 0}]{hyperref}
\usepackage{cleveref}

\usepackage{CJKutf8}

\usepackage{caption}
\usepackage{thmtools}

\usepackage{mathrsfs}

\crefname{theorem}{Theorem}{Theorems}
\crefname{thm}{Theorem}{Theorems}
\crefname{lemma}{Lemma}{Lemmas}
\crefname{lem}{Lemma}{Lemmas}
\crefname{remark}{Remark}{Remarks}
\crefname{prop}{Proposition}{Propositions}
\crefname{defn}{Definition}{Definitions}
\crefname{corollary}{Corollary}{Corollaries}
\crefname{conjecture}{Conjecture}{Conjectures}
\crefname{question}{Question}{Questions}
\crefname{chapter}{Chapter}{Chapters}
\crefname{section}{Section}{Sections}
\crefname{figure}{Figure}{Figures}
\crefname{example}{Example}{Examples}

\theoremstyle{plain}
\newtheorem{thm}{Theorem}[section]

\newtheorem{lemma}[thm]{Lemma}
\newtheorem{theorem}[thm]{Theorem}

\newtheorem{corollary}[thm]{Corollary}

\newtheorem{prop}[thm]{Proposition}
\newtheorem{conjecture}[thm]{Conjecture}

\theoremstyle{definition}

\theoremstyle{remark}
\newtheorem{remark}[thm]{Remark}

\numberwithin{equation}{section}

\renewcommand{\P}{\mathbb P}
\newcommand{\E}{\mathbb E}

\newcommand{\Z}{\mathbb Z}

\usepackage[margin=0.9in]{geometry}
\usepackage{tikz}
\usepackage{pgfplots}
\usepackage{extarrows}
\usepackage{titling}
\usepackage{commath}
\usepackage{wasysym}
\usepackage{mathtools}
\usepackage{titlesec}
\newcommand{\eps}{\varepsilon}
\usepackage{bbm}
\usepackage{setspace}
\usepackage{enumitem}
\usepackage{booktabs}

\usepackage[textsize=tiny]{todonotes}

\usepackage{cite}

\def\P{\mathbb{P}}

\DeclareMathSymbol{\leqslant}{\mathalpha}{AMSa}{"36} 
\DeclareMathSymbol{\geqslant}{\mathalpha}{AMSa}{"3E} 
\DeclareMathSymbol{\eset}{\mathalpha}{AMSb}{"3F}     

\renewcommand{\epsilon}{\varepsilon}

\makeatletter
\pgfdeclareshape{slash underlined}
{
  \inheritsavedanchors[from=rectangle] 
  \inheritanchorborder[from=rectangle]
  \inheritanchor[from=rectangle]{north}
  \inheritanchor[from=rectangle]{north west}
  \inheritanchor[from=rectangle]{north east}
  \inheritanchor[from=rectangle]{center}
  \inheritanchor[from=rectangle]{west}
  \inheritanchor[from=rectangle]{east}
  \inheritanchor[from=rectangle]{mid}
  \inheritanchor[from=rectangle]{mid west}
  \inheritanchor[from=rectangle]{mid east}
  \inheritanchor[from=rectangle]{base}
  \inheritanchor[from=rectangle]{base west}
  \inheritanchor[from=rectangle]{base east}
  \inheritanchor[from=rectangle]{south}
  \inheritanchor[from=rectangle]{south west}
  \inheritanchor[from=rectangle]{south east}
  \inheritanchorborder[from=rectangle]
  \foregroundpath{
    \southwest \pgf@xa=\pgf@x \pgf@ya=\pgf@y
    \northeast \pgf@xb=\pgf@x \pgf@yb=\pgf@y
    \pgf@xc=\pgf@xa
    \advance\pgf@xc by .5\pgf@xb
    \pgf@yc=\pgf@ya
    \advance\pgf@xc by -1.3pt
    \advance\pgf@yc by -1.8pt
    \pgfpathmoveto{\pgfqpoint{\pgf@xc}{\pgf@yc}}
    \advance\pgf@xc by  2.6pt
    \advance\pgf@yc by  3.6pt
    \pgfpathlineto{\pgfqpoint{\pgf@xc}{\pgf@yc}}
    \pgfpathmoveto{\pgfqpoint{\pgf@xa}{\pgf@ya}}
    \pgfpathlineto{\pgfqpoint{\pgf@xb}{\pgf@ya}}
 }
}
\tikzset{nomorepostaction/.code=\let\tikz@postactions\pgfutil@empty}

\makeatother

\makeatletter
\DeclareFontFamily{OMX}{MnSymbolE}{}
\DeclareSymbolFont{MnLargeSymbols}{OMX}{MnSymbolE}{m}{n}
\SetSymbolFont{MnLargeSymbols}{bold}{OMX}{MnSymbolE}{b}{n}
\DeclareFontShape{OMX}{MnSymbolE}{m}{n}{
    <-6>  MnSymbolE5
   <6-7>  MnSymbolE6
   <7-8>  MnSymbolE7
   <8-9>  MnSymbolE8
   <9-10> MnSymbolE9
  <10-12> MnSymbolE10
  <12->   MnSymbolE12
}{}
\DeclareFontShape{OMX}{MnSymbolE}{b}{n}{
    <-6>  MnSymbolE-Bold5
   <6-7>  MnSymbolE-Bold6
   <7-8>  MnSymbolE-Bold7
   <8-9>  MnSymbolE-Bold8
   <9-10> MnSymbolE-Bold9
  <10-12> MnSymbolE-Bold10
  <12->   MnSymbolE-Bold12
}{}

\let\llangle\@undefined
\let\rrangle\@undefined
\DeclareMathDelimiter{\llangle}{\mathopen}%
                     {MnLargeSymbols}{'164}{MnLargeSymbols}{'164}
\DeclareMathDelimiter{\rrangle}{\mathclose}%
                     {MnLargeSymbols}{'171}{MnLargeSymbols}{'171}
\makeatother

\pretitle{\begin{flushleft}\Large}
\posttitle{\par\end{flushleft}\vskip 0.5em
}
\preauthor{\begin{flushleft}}
\postauthor{
\\
\vspace{0.5em}
\footnotesize{The Division of Physics, Mathematics and Astronomy, California Institute of Technology\\ Email: 
\href{mailto:t.hutchcroft@caltech.edu}{t.hutchcroft@caltech.edu}}
\end{flushleft}}
\predate{\begin{flushleft}}
\postdate{\par\end{flushleft}}
\title{\bf Pointwise two-point function estimates and a non-pertubative proof of mean-field critical behaviour for long-range percolation} 

\renewenvironment{abstract}
 {\par\noindent\textbf{\abstractname.}\ \ignorespaces}
 {\par\medskip}

\titleformat{\section}
  {\normalfont\large \bf}{\thesection}{0.4em}{}

\author{{\bf Tom Hutchcroft}}

\begin{document}

\date{\small{\today}}

\maketitle

\setstretch{1.1}

\begin{abstract}
In long-range percolation on $\Z^d$, we connect each pair of distinct points $x$ and $y$ by an edge independently at random with probability  $1-\exp(-\beta\|x-y\|^{-d-\alpha})$, where $\alpha>0$ is fixed and $\beta\geq 0$ is a parameter. In a previous paper, we proved that if $0<\alpha<d$ then the critical two-point function satisfies the \emph{spatially averaged} upper bound
\[
\frac{1}{r^d}\sum_{x\in  [-r,r]^d} \P_{\beta_c}(0\leftrightarrow x) \preceq r^{-d+\alpha}
\]
for every $r\geq 1$. This upper bound is believed to be sharp for values of $\alpha$ strictly below the \emph{crossover value} $\alpha_c(d)$, and a matching lower bound  for $\alpha<1$ was proven by B\"aumler and Berger (AIHP 2022). In this paper, we prove \emph{pointwise} upper and lower bounds of the same order under the same assumption that $\alpha<1$. We also prove analogous two-sided pointwise estimates on the \emph{slightly subcritical} two-point function under the same hypotheses, interpolating between $\| x \|^{-d+\alpha}$ decay below the correlation length and $\| x \|^{-d-\alpha}$ decay above the correlation length. In dimensions $d=1,2,3$, we deduce that the triangle condition holds under the minimal assumption that $0<\alpha<d/3$. 
While this result had previously been established under additional perturbative assumptions using the lace expansion, our proof is completely non-perturbative and does not rely on the lace expansion in any way.
 In dimensions $1$ and~$2$ our results also treat the marginal case $\alpha=d/3$, implying that the triangle diagram diverges at most logarithmically  and hence that mean-field critical behaviour holds to within polylogarithmic factors.
\end{abstract}

\section{Introduction}
\label{sec:intro}

\textbf{Long-range percolation} on $\Z^d$ is defined to be the random graph with vertex set $\Z^d$ in which each potential edge $\{x,y\}$ is included independently at random with probability $1-\exp(-\beta J(x,y))$, where $J:\Z^d \times \Z^d \to [0,\infty)$ is a \textbf{symmetric kernel} (meaning that $J(x,y)=J(y,x)$ for every $x,y\in \Z^d$). We write $\P_\beta=\P_{\beta,J}$ for the law of the resulting random graph. We will be primarily interested in kernels that are \textbf{translation-invariant}, meaning that $J(x,y)=J(0,y-x)$ for every $x,y\in \Z^d$, and satisfy
\begin{equation}
J(x,y) \sim A \|x-y\|^{-d-\alpha} \qquad \text{ as $x-y\to\infty$}
\end{equation}
for some norm $\|\cdot\|$ on $\Z^d$ and some constants $A>0$ and $\alpha>0$; smaller values of $\alpha$ make longer edges more likely. Edges that are included in the random subgraph are referred to as \textbf{open}, and the connected components of the random subgraph are referred to as \textbf{clusters}.
 The model undergoes a phase transition
 in which an infinite cluster emerges as $\beta$ is varied through
  the \textbf{critical value}
\[
\beta_c=\beta_c(J):=\inf\{\beta \geq 0: \P_\beta(\text{an infinite cluster exists})>0 \},
\]
which is non-degenerate in the sense that $0<\beta_c<\infty$ if and only if $d\geq 2$ and $\alpha>0$ or $d=1$ and $0<\alpha \leq 1$ \cite{schulman1983long,newman1986one}. %
In this paper we study long-range percolation at and near the critical point $\beta_c$, referring the reader to 
e.g.\ \cite{heydenreich2015progress,biskup2021arithmetic,ding2023uniqueness,baumler2023distances} and references therein for other aspects of the model such as the geometry of supercritical clusters. It is also known that the model has a \emph{continuous} phase transition (with no infinite clusters at $\beta_c$) when $\alpha<d$ \cite{MR1896880,hutchcroft2020power} and a discontinuous phase transition when $d=\alpha=1$ \cite{MR868738}. It is conjectured that the phase transition is always continuous when $d\geq 2$, but the large-$\alpha$ problem is understood only in high dimensions \cite{sakai2018crossover,MR4032873,MR3306002,MR2430773}. Indeed, the large-$\alpha$ problem  appears to be similarly difficult as for nearest-neighbour models, where the analogous problem is understood only in two dimensions \cite{kesten1980critical,duminil2016absence} and high dimensions \cite{MR1043524,fitzner2015nearest}.

Both nearest-neighbour and long-range percolation are expected to exhibit a rich, fractal-like geometry at and near criticality. This geometry is described in part by \textbf{critical exponents}, with some of the most important  defined (conjecturally) by
\[
\P_{\beta_c}(0\leftrightarrow x) \approx \|x\|^{-d+2-\eta},\quad \P_{\beta_c}(|K|\geq n) \approx n^{-1/\delta},\quad \P_{\beta_c+\eps}(0\leftrightarrow \infty)\approx \eps^\beta,\; \text{ and } \; \E_{\beta_c-\eps}|K| \approx \eps^{-\gamma}.
\]
Here $\{0\leftrightarrow x\}$ denotes the event that $0$ and $x$ are connected (i.e., belong to the same cluster), $K$ denotes the cluster of the origin, and $\approx$ means that the ratio of the logarithms of both sides converges to $1$ in the appropriate limit; the point-to-point connection probability $\P_{\beta}(x\leftrightarrow y)$ is known as the \textbf{two-point function}. See \cite[Chapters 9 and 10]{grimmett2010percolation} for an introduction to these exponents, including both the rigorous and non-rigorous theory.
 The \emph{universality principle} predicts that exponents should depend on  $d$ and  $\alpha$ but not on the microscopic details of the model such as the exact choice of kernel. We stress that even proving the \emph{existence} of these critical exponents is a challenging open problem in many cases:
For nearest-neighbour percolation on $\Z^d$, almost nothing is known outside of the two-dimensional \cite{MR879034,smirnov2001critical,smirnov2001critical2,lawler2002one,nienhuis1987coulomb} and high-dimensional \cite{MR2239599,MR1043524,fitzner2015nearest,MR762034,MR1127713,MR2551766,MR2748397} settings. We will focus on the high dimensional case, which is more relevant to our results.

When $d$ is sufficiently large or $\alpha$ is sufficiently small (making the model's ``effective dimension'' large), the model is expected to exhibit \emph{mean-field critical behaviour}, meaning in particular that $\eta=\max\{0,2-\alpha\}$, $\delta=2$, $\beta=1$, and $\gamma=1$. More precisely, mean-field critical behaviour is expected to hold whenever $d$ is strictly larger than the \textbf{upper critical dimension} $d_c=\min\{6,3\alpha\}$, with logarithmic corrections to mean-field scaling when $d=d_c$.
 Until recently, the only way to prove mean-field critical behaviour in finite dimension was via the \emph{lace expansion}, a method introduced by Brydges and Spencer \cite{MR782962} to study weakly self-avoiding walk and first applied to percolation in the landmark work of Hara and Slade \cite{MR1043524}.
    The lace expansion was applied to long-range percolation in the work of Heydenreich, van der Hofstad, and Sakai \cite{MR2430773} and Chen and Sakai \cite{MR3306002,MR4032873}, 
   who prove in particular that if $d>\min\{6,3\alpha\}$ and the kernel $J$ is invariant under permutations of coordinates and reflections in coordinate hyperplanes and satisfies a certain perturbative criterion (i.e., is sufficiently `spread-out') then
 \begin{equation}
\P_{\beta_c}(0 \leftrightarrow x) \asymp \begin{cases} 
\|x\|^{-d+\alpha} & \alpha<2\\
 \|x\|^{-d+2} \frac{1}{\log \|x\|} & \alpha = 2\\
 \|x\|^{-d+2} & \alpha > 2
\end{cases}
\qquad \text{ as $\|x\|\to\infty$.}
\end{equation}
In other words, the two-point function is of the same order as the Green's function for the random walk with jump kernel proportional to $\|x-y\|^{-d-\alpha}$, which converges to Brownian motion when $\alpha>2$ and to a symmetric $\alpha$-stable L\'evy process when $\alpha<2$.
Notably, this analysis also works in the marginal case $d=6$, $\alpha=2$, where the logarithmic correction to the two-point function (which is also present in the random walk Green's function when $\alpha=2$) causes the model to have exact mean-field scaling \cite{MR4032873}. The analogous ``$x$-space" two-point function bounds for high-dimensional nearest-neighbour percolation were first proven by Hara, van der Hofstad, and Slade \cite{MR1959796}, with the original approach to the lace expansion working with Fourier coefficients.
These two-point function estimates imply under the same assumptions that the \emph{triangle diagram}
 \[
\nabla_\beta :=\sum_{x,y\in \Z^d} \P_\beta(0\leftrightarrow x)\P_\beta(x\leftrightarrow y)\P_\beta(y\leftrightarrow 0)
\]
converges at $\beta_c$, which is a sufficient condition for various other exponents to take their mean-field values \cite{MR762034,MR1127713,HutchcroftTriangle}.
An interesting feature of Sak's prediction is that when $d<6$ there is a non-trivial regime $(d/3,\alpha_c)$ in which the model is not mean-field but the simple scaling of the two-point function $\P_{\beta_c}(0\leftrightarrow x)\approx \|x\|^{-d+\alpha}$ continues to hold.

As alluded to above, a major shortcoming of the lace expansion is that it is an inherently \emph{perturbative} method, meaning that it needs a ``small parameter'' to work and cannot be applied using only the minimal assumption that $d>d_c$. In practice, this usually means either taking the dimension to be very large or using a ``spread out'' kernel that is roughly constant over a ball of large radius. It remains a major open problem to prove mean-field behaviour for nearest-neighbour percolation on Euclidean lattices under the minimal assumption that $d>6$, with the current best results for the standard hypercubic lattice $\Z^d$ requiring $d \geq 11$ \cite{fitzner2015nearest}. Moreover, the lace expansion seems totally unsuited to analyzing percolation at the upper-critical dimension itself, where logarithmic corrections to mean-field scaling are expected \cite{essam1978percolation}.

The main result of this paper is a non-perturbative proof of mean-field critical behaviour for $\alpha<\min\{d/3,1\}$ and of mean-field critical behaviour up to polylogarithmic factors in the marginal case $\alpha=d/3$ when $d\in \{1,2\}$. 
The proof does not rely on the lace expansion in any way. 

\begin{theorem}
\label{cor:triangle}
Let $d\geq 1$, let $J:\Z^d\times \Z^d \to [0,\infty)$ be a symmetric, integrable, translation-invariant kernel, and suppose that there exist constants $0<\alpha<d$ and $c>0$ such that $c \|x-y\|^{-d-\alpha} \leq J(x,y) \leq C\|x-y\|^{-d-\alpha}$ for all distinct $x,y \in \Z^d$.  \begin{enumerate}
\item 
If $\alpha<\min\{1,d/3\}$ then $\nabla_{\beta_c}<\infty$. As a consequence, the model has mean-field critical behaviour in the sense that
\[
\P_{\beta_c}(|K|\geq n) \asymp \frac{1}{\sqrt{n}}, \qquad \E_{\beta_c-\eps}|K| \asymp \frac{1}{\eps}, \quad \text{ and } \quad \P_{\beta_c+\eps}(|K|=\infty) \asymp \eps
\]
as $n\to\infty$ or $\eps \downarrow 0$ as appropriate.
\item If $d\in \{1,2\}$ and $\alpha = d/3$ then the triangle diagram diverges at most logarithmically as $\beta\uparrow \beta_c$ in the sense that $\nabla_{\beta_c-\eps}=O(\log 1/\eps)$ as $\eps \downarrow 0$. As a consequence, the model has mean-field critical behaviour to within polylogarithmic factors 
in the sense that
\[
\P_{\beta_c}(|K|\geq n) \asymp \frac{(\log n)^{O(1)}}{\sqrt{n}}, \quad \E_{\beta_c-\eps}|K| \asymp \frac{(\log \frac{1}{\eps})^{O(1)}}{\eps}, \; \text{ and } \; \P_{\beta_c+\eps}(|K|=\infty) \asymp  (\log \frac{1}{\eps})^{O(1)}\eps
\]
as $n\to\infty$ or $\eps \downarrow 0$ as appropriate. Moreover, the exponent of the logarithm is bounded above by $2$ in each case.
\end{enumerate}
\end{theorem}

The fact that logarithmic divergence of the triangle diagram implies that mean-field critical behaviour holds to within polylogarithmic factors was proven in \cite{HutchcroftTriangle}. The notational conventions used in the statement of this theorem, and throughout the rest of the paper, are as follows:
We write $\asymp$, $\preceq$, and $\succeq$ to denote equalities and inequalities that holds to within positive multiplicative constants depending only on the parameters $d$, $\alpha$, $C$, and $c$ appearing in the statement of our main theorems. (Sometimes the constants will depend only on some subset of these parameters as appropriate if not all appear in the statement being proven.) Big-$O$ notation will be used similarly, so that all implicit constants depend only on $d$, $\alpha$, $C$, and $c$. Throughout the rest of the paper, we will let $\|\cdot\|=\|\cdot\|_\infty$ denote the $\ell^\infty$ norm on $\Z^d$. (Note that all our main theorems are insensitive to the choice of norm on $\Z^d$; the $\ell^\infty$ norm is merely the most notationally convenient for our proofs.) To avoid dividing by zero, we use the notation $\langle x-y\rangle=\max\{2,\| x-y\|\}$.

\begin{remark}
It is an amusing feature of our methods that they are able to prove mean-field behaviour of long-range percolation in \emph{low dimensions} $d\in \{1,2,3\}$ under the optimal assumption that $\alpha<d/3$, but give only partial results in higher dimensions. This is in stark contrast to the lace expansion, which typically works more efficiently as the dimension gets higher.
\end{remark}

The proof of \cref{cor:triangle} is closely related to analogous developments for long-range percolation on the \emph{hierarchical} lattice, where we now have a fairly detailed understanding of critical behaviour in all three regimes $d>d_c$, $d=d_c$, and $d<d_c$ \cite{hutchcrofthierarchical,hutchcroft2022critical}. We conjecture that long-range percolation with $d<6$ and $\alpha=d/3$ has the same logarithmic corrections to scaling as the hierarchical model, some of which were computed exactly in \cite{hutchcroft2022critical}, and which are \emph{not} the same corrections predicted to hold for the nearest-neighbour model in six dimensions \cite{essam1978percolation}. Here is one concrete instance of this conjecture:

\begin{conjecture}
Let $d\geq 1$, let $J:\Z^d\times \Z^d \to [0,\infty)$ be a symmetric, integrable, translation-invariant kernel, and suppose that there exist constants $0<\alpha<d$ and $c>0$ such that $c \|x-y\|^{-d-\alpha} \leq J(x,y) \leq C\|x-y\|^{-d-\alpha}$ for all distinct $x,y \in \Z^d$. If $d<6$ and $\alpha=d/3$ then the logarithmic corrections to scaling are the same as for the hierarchical lattice. In particular,
\[
\P_{\beta_c}(|K|\geq n) \asymp \frac{(\log n)^{1/4}}{\sqrt{n}}
\]
as $n\to \infty$.
\end{conjecture}

\medskip

\noindent \textbf{Sak's prediction and pointwise estimates on the critical two-point function.} 
The proof of \cref{cor:triangle} builds on a recent series of works studying long-range percolation \emph{beyond} the mean-field regime \cite{hutchcroft2020power,hutchcroft2022sharp,baumler2022isoperimetric}, and in fact \cref{cor:triangle} will be a simple corollary of more general theorems establishing pointwise estimates on critical and near-critical two-point functions for long-range percolation with $\alpha<1$. Before stating these theorems, let us explain the relevant context.


In contrast to the infamously mysterious nature of other aspects of intermediate-dimensional critical phenomena,
 there is a surprisingly simple prediction for the dependence of the two-point exponent $\eta$ on the long-range parameter $\alpha$ when the dimension $d$ is fixed:
\begin{equation}
\label{eq:eta_prediction}
2-\eta = \begin{cases} \alpha & \alpha \leq \alpha_c\\
2-\eta_\mathrm{SR} & \alpha>\alpha_c,
\end{cases}
\end{equation}
where $\eta_\mathrm{SR}$ is the value of the exponent $\eta$ for the nearest-neighbour model in $d$ dimensions and the \textbf{crossover value} $\alpha_c=2-\eta_\mathrm{SR}$ is the unique value making this function continuous. For $d=1$, the analogous conjecture is that the equality $2-\eta=\alpha$ for every $\alpha\in (0,1)$.
This prediction is essentially due to Sak \cite{sak1973recursion}, who studied the Ising model rather than percolation; see also e.g.\ \cite{brezin2014crossover,behan2017scaling,luijten1997interaction,gori2017one,behan2017scaling} for further discussion in the physics literature (where there has been some controversy over whether Sak's prediction is correct \cite{van1982instability,picco2012critical,brezin2014crossover}) and \cite{MR3772040,MR3723429} for rigorous partial progress on the spin $O(n)$ model.

Significant progress on Sak's prediction for long-range percolation beyond the mean-field regime has been made in the recent works \cite{hutchcroft2022sharp,baumler2022isoperimetric}. First, in \cite{hutchcroft2022sharp}, we proved that if $J:\Z^d\times \Z^d\to [0,\infty)$ is a symmetric, translation invariant kernel and there exist $0<\alpha<d$ and a constant $c$ such that $J(x,y)\geq c\|x-y\|^{-d-\alpha}$ for every distinct $x,y\in \Z^d$ then 
\begin{equation}
\label{eq:spatially_averaged}
\frac{1}{r^d}\sum_{x\in [-r,r]^d} \P_{\beta_c}(0\leftrightarrow x) \preceq r^{-d+\alpha}
\end{equation}
for every $r\geq 1$. In particular, the critical exponent $\eta$ satisfies $2-\eta \leq \alpha$ whenever it is well-defined.
Following this work, B\"aumler and Berger \cite{baumler2022isoperimetric} proved that a matching lower bound holds when $J(x,y)\leq C\|x-y\|^{-d-\alpha}$ and $\alpha<1$. In one dimension this covers the entire range of relevant values, proving Sak's prediction that $2-\eta=\alpha$ for all $0<\alpha<1$ when $d=1$. One can think of the results of \cite{baumler2022isoperimetric} as showing that long-range, ``bulk-to-bulk'' effects always dominate short-range, ``boundary-to-boundary'' effects when $\alpha<1$, which is related to the inequality $\alpha_c\geq 1$. (Conditional on the validity of Sak's prediction, this is equivalent to the bound $\eta_\mathrm{SR}\leq 1$, which is known to hold in every dimension \cite[Lemma 3.1]{MR2748397}. In fact $\alpha_c$ is equal to $2$ in high dimensions and has a numerical value close to $2$ in every dimension $d>1$.)  Both works \cite{hutchcroft2022sharp,baumler2022isoperimetric} followed aforementioned earlier work for long-range percolation on the \emph{hierarchical lattice} \cite{hutchcrofthierarchical}, where we still have a much more refined understanding of critical behaviour than for long-range percolation on $\Z^d$ \cite{hutchcroft2022critical}.

The most fundamental result of this paper is an improvement of the spatially-averaged estimates of \cite{hutchcroft2022sharp,baumler2022isoperimetric} to \emph{pointwise} two-sided estimates under the assumption that $\alpha<1$.

\begin{theorem}[The critical two-point function]
\label{thm:main_pointwise}
Let $d\geq 1$, let $J:\Z^d\times \Z^d \to [0,\infty)$ be a symmetric, integrable, translation-invariant kernel, and suppose that there exist constants $0<\alpha<d$ and $c>0$ such that $c \|x-y\|^{-d-\alpha} \leq J(x,y) \leq C\|x-y\|^{-d-\alpha}$ for all distinct $x,y \in \Z^d$. If $\alpha<1$ then 
\[
\P_{\beta_c}(0\leftrightarrow x) \asymp  \|x\|^{-d+\alpha}
\] 
for every $x\in \mathbb{Z}^d \setminus \{0\}$. 
\end{theorem}

This \emph{upper bound} of this theorem is proven via an induction on scales, where we use the \emph{random box exit decomposition} method of \cite{baumler2022isoperimetric} to improve the spatially-averaged upper bounds of \cite{hutchcroft2022sharp} to pointwise estimates. It is interesting to note that this method was originally developed for proving \emph{lower bounds} of the form $\|x\|^{-d+\alpha}$ on the two-point function, which should only hold for $\alpha<\alpha_c$, while the pointwise upper bound of \cref{thm:main_pointwise} is presumably true for all $\alpha>0$; proving this remains a significant open problem. The \emph{lower bound} of \cref{thm:main_pointwise} is proven via a second moment analysis of the number of long edges belonging to a certain kind of open path from $0$ to $x$, where we use the results of \cite{baumler2022isoperimetric} to lower bound the first moment and the results of \cite{hutchcroft2020power} to upper bound the second moment.
%

\medskip

\noindent 
\textbf{The slightly subcritical two-point function}. In order to prove logarithmic divergence of the triangle diagram at the critical dimension $d=d_c$, we will also need control of the two-point function in the near-critical regime, when $\beta=\beta_c-\eps$. While fairly weak estimates would suffice for this application, we will in fact be able to prove sharp bounds that are of independent interest.

Let us first recall what is known for the subcritical two-point function when $\beta<\beta_c$ is fixed and $x\to \infty$.
Under mild regularity assumptions on the kernel $J$ (which hold when, say, $J(x,y)\sim C\llangle x-y\rrangle^{-d-\alpha}$ for some norm $\llangle \cdot \rrangle$ on $\Z^d$ and some $\alpha>0$), it is known that
\[
\P_\beta(x\leftrightarrow y) \sim \beta \chi(\beta)^2 J(x,y) \qquad \text{as $x-y\to \infty$ for every fixed $0<\beta < \beta_c$,}
\]
where $\chi(\beta) := \E_\beta|K|= \sum_{x\in \Z^d} \P_\beta(0\leftrightarrow x)$ is the \emph{susceptibility}, which is finite if and only if $\beta<\beta_c$ by the sharpness of the phase transition \cite{aizenman1987sharpness,duminil2015new,1901.10363}. As explained in detail in \cite{MR4248721}, one can interpret this theorem as stating that the most efficient way to connect $0$ to $x$ is for a single long edge to join ``typical small clusters'' around $0$ and $x$, with each such cluster having expected size $\chi(\beta)$. 

Our next theorem establishes sharp upper bounds on the slightly subcritical two-point function that interpolate between the critical decay $\|x-y\|^{-d+\alpha}$ and the subcritical decay $\chi(\beta)^2\|x-y\|^{-d-\alpha}$ in the near-critical regime under the assumption that $\alpha<1$. The transition between these two regimes occurs at the
\textbf{correlation length} $\xi(\beta)$, which is defined formally for $0\leq \beta \leq \beta_c$ by
\[
\xi(\beta):= \inf\left\{r \geq 1: \sum_{x\in [-r,r]^d} \P_\beta(0\leftrightarrow x) \geq \frac{1}{2}\chi(\beta) \right\},
\]
so that $\xi(\beta_c)=\infty$. 

\begin{theorem}[The slightly subcritical two-point function]
\label{thm:subcritical}
Let $d\geq 1$, let $J:\Z^d\times \Z^d \to [0,\infty)$ be a symmetric, integrable, translation-invariant kernel, and suppose that there exist constants $0<\alpha<d$ and $c>0$ such that $c \|x-y\|^{-d-\alpha} \leq J(x,y) \leq C\|x-y\|^{-d-\alpha}$ for all distinct $x,y \in \Z^d$. If $\alpha<1$ then $\chi(\beta)\asymp \xi(\beta)^\alpha$ for every $0\leq \beta \leq \beta_c$ and 
\[
\P_{\beta}(0\leftrightarrow x) \asymp \Biggl\{\begin{matrix} \|x\|^{-d+\alpha} 
& \text{ if }\|x\|\leq \xi(\beta) \\
\chi(\beta)^2 \|x\|^{-d-\alpha} & \text{ if }\|x\|> \xi(\beta) 
\end{matrix}\Biggr\}
\asymp \|x\|^{-d+\alpha} \left(1 \vee \frac{\|x\|}{\xi(\beta)}\right)^{-2\alpha}
\] 
for every $\beta_c/2\leq \beta \leq \beta_c$ and $x\in \mathbb{Z}^d \setminus\{0\}$. 
\end{theorem}

\begin{remark}
The fact that $\chi(\beta)\asymp \xi(\beta)^\alpha$ when $\alpha<1$ is an easy consequence of the results of \cite{hutchcroft2022sharp,baumler2022isoperimetric}; see \cref{lem:correlation_length} below. Since $2-\eta=\alpha$ in the same regime, it may be thought of as a rigorous instance of Fisher's relation $\gamma=(2-\eta)\nu$ \cite[Eq.\ 9.21]{grimmett2010percolation}, where $\nu$ is the correlation length exponent $\xi(\beta)\approx |\beta-\beta_c|^{-\nu}$.
\end{remark}

\begin{remark}
Similar sharp upper bounds on the slightly subcritical two-point function for \emph{nearest-neighbour} percolation in high dimensions were proven by the author, Michta, and Slade in \cite{hutchcroft2023high}; in that case one interpolates between $\|x-y\|^{-d+2}$ decay and exponential decay around a correlation length of order $|p-p_c|^{-2}$. It remains open to prove a matching lower bound. The methods of \cite{hutchcroft2023high} are completely different to those of this paper, and rely on estimates for critical percolation proven via the lace expansion including the half-space two-point function estimates of \cite{chatterjee2020restricted}.
\end{remark}

\section{Proof}

We will assume throughout the paper that the reader is familiar with the FKG and BK inequalities and the attendant notion of the \emph{disjoint occurence} of two increasing events, referring them to \cite[Chapter 2.3]{grimmett2010percolation} otherwise.

\subsection{Pointwise upper bounds on the critical two-point function}

In this section we prove the following proposition, and hence the upper bound of \cref{thm:main_pointwise}.

\begin{prop}
\label{prop:main_pointwise_upper}
Let $d\geq 1$, let $J:\Z^d\times \Z^d \to [0,\infty)$ be a symmetric, integrable, translation-invariant kernel, and suppose that there exist constants $0<\alpha<d$ and $c>0$ such that $c \|x-y\|^{-d-\alpha} \leq J(x,y) \leq C\|x-y\|^{-d-\alpha}$ for all distinct $x,y \in \Z^d$. If $\alpha<1$ then
\[
\P_{\beta_c}(0\leftrightarrow x) \preceq \|x\|^{-d+\alpha}
\] 
for every $x\in \mathbb{Z}^d$. 
\end{prop}

 As mentioned in the introduction, the proof relies on a method introduced by B\"aumler and Berger~\cite{baumler2022isoperimetric} in which we consider the first edge an open path uses to exit a \emph{randomly translated box}. (In fact \cite{baumler2022isoperimetric} uses a box of \emph{random radius}; the difference is not important.) This method is encapsulated by the following lemma. For each $k\geq 0$, we write $\Lambda_k = [-2^k,2^k]^d=\{x\in \Z^d:\|x\|\leq 2^k\}$ for the box of side length $2^{k+1}+1$ centred at the origin. Recall that we write $\langle x \rangle=\max\{2,\|x\|\}$.

\begin{lemma}[Random box exit decomposition]
\label{lem:random_box_convolution}
Let $d\geq 1$, let $J:\Z^d\times \Z^d \to [0,\infty)$ be a symmetric, integrable, translation-invariant kernel, and suppose that there exist constants $0<\alpha<d$ and $C<\infty$ such that $J(x,y) \leq C\|x-y\|^{-d-\alpha}$ for all distinct $x,y \in \Z^d$. If $\alpha<1$ then
\begin{equation}
\P_\beta(0 \leftrightarrow x) \preceq \beta \sum_{u\in \Lambda_{k-1}} \sum_{v\in \Z^d \setminus \Lambda_{k-3}} \langle u-v\rangle^{-d-\alpha} \P_\beta(0\leftrightarrow u) \P_\beta(v\leftrightarrow x) \min\left\{1,\frac{\langle u-v\rangle}{2^k} \right\}
\label{eq:random_box_convolution}
\end{equation}
for every $k\geq 3$, $\beta>0$, and $x\in \Z^d \setminus \Lambda_{k-1}$.
\end{lemma}

\begin{proof}[Proof of \cref{lem:random_box_convolution}]
Fix $k\geq 3$ and $x\in \Z^d$ with $2^{k-1}<\|x\|\leq 2^k$. For each $z\in \Z^d$, let $\Lambda^z_{k-2}=z+\Lambda_{k-2}=\{y\in \Z^d: \|y-z\| \leq 2^{k-2}\}$ be the box of sidelength $2^{k-1}+1$ centered at $z$, so that if $z\in \Lambda_{k-3}$ then 
\[\Lambda_{k-3} \subseteq \Lambda^z_{k-2} \subseteq \Lambda_{k-1}.\] In particular, if $z\in \Lambda_{k-3}$ then $\Lambda^z_{k-2}$ contains $0$ and does not contain $x$, so that if $0$ is connected to $x$ by some simple open path $\gamma$ then there must exist distinct vertices $u\in \Lambda^z_{k-2}$ and $v\in \Z^d\setminus \Lambda^z_{k-2}$ such that the event $\{0\leftrightarrow u\} \circ \{\{u,v\} \text{ open}\} \circ \{v \leftrightarrow x\}$ occurs. (Indeed, this holds with $\{u,v\}$ equal to the edge crossed by $\gamma$ as it leaves the box for the first time.) Applying a union bound and the BK inequality, we deduce that
\[
\P_\beta(0 \leftrightarrow x) \leq \sum_{u,v\in \Z^d} \mathbbm{1}(u\in \Lambda^z_{k-2},v\notin \Lambda^z_{k-2}) \beta \langle u-v\rangle^{-d-\alpha} \P_\beta(0\leftrightarrow u) \P_\beta(v\leftrightarrow x)
\]
for every $z\in \Lambda_{k-3}$. 
 Now, observe that in order for the indicator $\mathbbm{1}(u\in \Lambda_z,v\notin \Lambda_z)$ to be one, there must exist a coordinate $i\in \{1,\ldots,d\}$ such that $u_i \leq z_i \leq v_i$ or $v_i \leq z_i \leq u_i$, so that
\[
\frac{1}{|\Lambda_{k-3}|} \sum_{z\in \Lambda_{k-3}} \mathbbm{1}(u\in \Lambda_z,v\notin \Lambda_z) \preceq \min\left\{1, \frac{\langle u-v\rangle}{2^k}\right\}.
\]
Thus, averaging the above inequality over $z\in \Lambda_{k-2}$, we obtain that
\begin{equation*}
\P_\beta(0 \leftrightarrow x) \preceq \sum_{u\in \Lambda_{k-1}} \sum_{v\in \Z^d \setminus \Lambda_{k-3}} \langle u-v\rangle^{-d-\alpha} \P_\beta(0\leftrightarrow u) \P_\beta(v\leftrightarrow x) \min\left\{1,\frac{\langle u-v\rangle}{2^k} \right\},
\end{equation*}
 where the restrictions $u\in \Lambda_{k-1}$ and $v\notin \Lambda_{k-3}$ are required for $u$ to belong to any of the boxes $\Lambda^z_{k-2}$ with $z\in \Lambda_{k-2}$. 
\end{proof}

The rest of the proof is purely analytic, showing that the two inequalities \eqref{eq:spatially_averaged} and \eqref{eq:random_box_convolution} imply a pointwise version of \eqref{eq:spatially_averaged} when $\alpha<1$.

\begin{proof}[Proof of \cref{prop:main_pointwise_upper}]
For each $k\geq 0$ let $A_k$ be minimal such that
\[
\P_{\beta_c}(0\leftrightarrow x) \leq A_k \langle x\rangle^{-d+\alpha} \qquad \text{ for every $x\in \Lambda_k$.}
\]
We claim that there exists a constant $C_0=C_0(d,\alpha,C,c)$ such that
\[
A_{k+1} \leq \max\left\{A_k, C_0 + \frac{A_k}{2}\right\}
\]
for every $k\geq 3$.
Once this is established, it will follow by induction on $k$ that
\[
A_\infty:=\sup_k A_k \leq \max\{A_3,2C_0\}<\infty
\]
as required.  

\medskip

Fix $k\geq 3$ and $x\in \Z^d$ with $2^{k-1}<\|x\|\leq 2^k$.
We will analyze \eqref{eq:random_box_convolution} by breaking the sum into three parts: One involving points $u$ and $v$ that are well-separated from each other and have distance $O(2^k)$ from the origin, which we will bound using \eqref{eq:spatially_averaged}, one involving points $u$ and $v$ that are close to each other, which we will bound in terms of $A_{k-1}$, and a final term involving points $v$ that are very far from both the origin and $x$, which we will bound using \eqref{eq:random_box_convolution}. We will introduce a parameter $\eps$ controlling the scale of closeness between $u$ and $v$ at which we split between the first two sums, which we will eventually take to be a small constant. We now do this precisely.
Let $0<\eps \leq 2^{-4}$ be a small parameter to be optimized over later, and consider the three sums
\begin{align*}
\textsc{I}(\eps) &:= 2^{-k} \sum_{u\in \Lambda_{k-1}} \sum_{v \in \Lambda_k} \mathbbm{1}(\|u-v\|\geq \eps 2^k) \langle u-v\rangle^{1-d-\alpha} \P_\beta(0\leftrightarrow u) \P_\beta(v\leftrightarrow x)\\
\textsc{II}(\eps) &:= 2^{-k} \sum_{u\in \Lambda_{k-1}} \sum_{v\in \Lambda_k \setminus \Lambda_{k-3}} \mathbbm{1}(1\leq \|u-v\|< \eps 2^k) \langle u-v\rangle^{1-d-\alpha} \P_\beta(0\leftrightarrow u) \P_\beta(v\leftrightarrow x)\\
\textsc{III} &:= \sum_{u\in \Lambda_{k-1}} \sum_{v \in \Z^d \setminus \Lambda_k}  \langle u-v\rangle^{-d-\alpha} \P_\beta(0\leftrightarrow u) \P_\beta(v\leftrightarrow x),
\end{align*}
so that
 \[\P_\beta(0\leftrightarrow x) \preceq \textsc{I}(\eps) +  \textsc{II}(\eps) + \textsc{III}.\]
 (It is important both here and in the rest of the proof that none of the implicit constants depend on the choice of $0<\eps \leq 1/16$.)
For the first term, we apply \eqref{eq:spatially_averaged} to obtain that
\begin{multline}
\textsc{I}(\eps) \leq 2^{-k} (\eps 2^k)^{1-d-\alpha} \sum_{u\in \Lambda_{k-1}} \sum_{v\in \Lambda_k} \P_\beta(0\leftrightarrow u) \P_\beta(v\leftrightarrow x) 
\\\preceq \eps^{1-d-\alpha} 2^{-(d+\alpha)k} (2^{\alpha k})^2 
= \eps^{1-d-\alpha} 2^{-(d-\alpha)k}.
\label{eq:I}
\end{multline}
For the second term, we note that if $v\notin \Lambda_{k-3}$ and $\|u-v\| < \eps 2^k \leq 2^{k-4}$ then $u \notin \Lambda_{k-4}$. Using the bound $\P_\beta(0\leftrightarrow u) \leq A_{k-1} \|u\|^{-d+\alpha} \preceq A_{k-1} 2^{-(d-\alpha)k}$ we  deduce that
\[
\textsc{II}(\eps) \leq A_{k-1} 2^{-(d+1-\alpha)k} \sum_{u\in \Z^d} \sum_{v\in\Lambda_k} \mathbbm{1}(1\leq \|u-v\|< \eps n) \langle u-v\rangle^{1-d-\alpha} \P_\beta(v\leftrightarrow x).
\]
By considering the number of points in the boundary of each box around $v$, we can bound the sum over $u$ appearing here as
\[
\sum_{u \in \Z^d\setminus \{v\}}   \mathbbm{1}(1\leq \|u-v\|< \eps 2^k) \langle u-v\rangle^{1-d-\alpha} \preceq \sum_{r=1}^{\lceil \eps 2^k \rceil} r^{d-1}r^{1-d-\alpha}  \preceq \eps^{1-\alpha} 2^{(1-\alpha)k},
\]
where we used the assumption that $\alpha<1$ in the final estimate. Plugging this into the upper bound for $\textsc{II}(\eps)$ above, we obtain that
\begin{equation}
\label{eq:II}
\textsc{II}(\eps) \preceq A_{k-1} \eps^{1-\alpha} 2^{-dk}  \sum_{v\in \Lambda_k} \P_\beta(x \leftrightarrow v) \preceq  A_{k-1} \eps^{1-\alpha} 2^{-(d-\alpha)k}.
\end{equation}
  Finally, we can bound 
\begin{align}
\textsc{III} &\leq \sum_{u\in \Lambda_{k-1}} \sum_{\ell=1}^\infty \sum_{v\in \Lambda_{k+\ell} \setminus \Lambda_{k+\ell-1}} \langle u-v\rangle^{-d-\alpha} \P_\beta(0\leftrightarrow u) \P_\beta(v\leftrightarrow x) \nonumber\\
&\preceq \sum_{u\in \Lambda_{k-1}} \sum_{\ell=1}^\infty \sum_{v\in \Lambda_{k+\ell} \setminus \Lambda_{k+\ell-1}} 2^{-(d+\alpha)(k+\ell)} \P_\beta(0\leftrightarrow u) \P_\beta(v\leftrightarrow x) \nonumber\\
&\preceq 2^{\alpha k} \sum_{\ell =1}^\infty 2^{-(d-\alpha)(k+\ell)}  \cdot 2^{\alpha(k+\ell)} \preceq 2^{-(d-\alpha)k}. \label{eq:III}
\end{align}
Putting these three estimates together, we obtain that there exists a constant $C_1=C_1(d,\alpha,C,c)\geq 1$ such that
\[
\P_\beta(0\leftrightarrow x) \leq C_1 \left[\eps^{1-d-\alpha}+\eps^{1-\alpha} A_{k-1} + 1 \right] \langle x\rangle^{-d+\alpha}
\]
for every $0<\eps\leq 1/16$ and $x\in \Lambda_k \setminus \Lambda_{k-1}$. It follows from the definitions that
\[
A_k \leq \max\left\{A_{k-1},C_1 \left[\eps^{1-d-\alpha}+\eps^{1-\alpha} A_{k-1} + 1 \right] \right\},
\]
for every $0\leq \eps \leq 1/16$ and $k\geq 0$, and the claim follows by taking $0<\eps \leq 1/16$ sufficiently small that $C_1 \eps^{1-\alpha} \leq 1/2$.
\end{proof}

\subsection{Upper bounds on the slightly subcritical two-point function}

In this section we prove the following proposition, which yields the upper bound of \cref{thm:subcritical}.

\begin{prop}[The slightly subcritical two-point function]
\label{prop:subcritical_upper}
Let $d\geq 1$, let $J:\Z^d\times \Z^d \to [0,\infty)$ be a symmetric, integrable, translation-invariant kernel, and suppose that there exist constants $0<\alpha<d$ and $c>0$ such that $c \|x-y\|^{-d-\alpha} \leq J(x,y) \leq C\|x-y\|^{-d-\alpha}$ for all distinct $x,y \in \Z^d$. If $\alpha<1$ then $\chi(\beta)\asymp \xi(\beta)^\alpha$ for every $0\leq \beta \leq \beta_c$ and 
\[
\P_{\beta}(0\leftrightarrow x) \preceq \Biggl\{\begin{matrix} \|x\|^{-d+\alpha} 
& \text{ if }\|x\|\leq \xi(\beta) \\
\chi(\beta)^2 \|x\|^{-d-\alpha} & \text{ if }\|x\|> \xi(\beta) 
\end{matrix}\Biggr.
\] 
for every $0<\beta \leq \beta_c$ and $x\in \mathbb{Z}^d$. 
\end{prop}

  This proposition will be used to prove logarithmic divergence of the triangle diagram for $\alpha =d/3$ in the next subsection. We begin by noting how the relation $\chi(\beta)\asymp \xi(\beta)^\alpha$ follows from the results of \cite{hutchcroft2022sharp,baumler2022isoperimetric}.

\begin{lemma}
\label{lem:correlation_length}
Let $d\geq 1$, let $J:\Z^d\times \Z^d \to [0,\infty)$ be a symmetric, integrable, translation-invariant kernel, and suppose that there exist constants $0<\alpha<d$ and $C<\infty$ such that $J(x,y) \leq C\|x-y\|^{-d-\alpha}$ for all distinct $x,y \in \Z^d$. If $\alpha<1$ then there exists a constant $\eps_0=\eps_0(d,\alpha,C)$ such that the implication
\[
\left[\sum_{x\in [-n,n]^d} \P_\beta(0\leftrightarrow x) \leq \frac{\eps_0}{\beta} n^{\alpha} \right] \Rightarrow 
\left[\chi(\beta) \leq 2\sum_{x\in [-n,n]^d} \P_\beta(0\leftrightarrow x) \right]
\]
holds for all $\beta>0$ and $n\geq 1$.
\end{lemma}

\begin{proof}
This is a consequence of \cite[Lemma 2.5]{baumler2022isoperimetric} together with a standard application of the ``$\phi_\beta(S)$ argument'' of \cite{duminil2015new} as explained in detail in \cite[Section 4]{HutchcroftTriangle}.
\end{proof}

\begin{corollary}
\label{cor:susceptibility_correlation_length}
Let $d\geq 1$, let $J:\Z^d\times \Z^d \to [0,\infty)$ be a symmetric, integrable, translation-invariant kernel, and suppose that there exist constants $0<\alpha<d$ and $C<\infty$ such that $J(x,y) \leq C\|x-y\|^{-d-\alpha}$ for all distinct $x,y \in \Z^d$. If $\alpha<1$ then $\chi(\beta)\asymp \xi(\beta)^\alpha$ for every $0\leq \beta <\beta_c$.
\end{corollary}

We now turn to the pointwise upper bound of \cref{thm:subcritical,prop:subcritical_upper}. 
Note that \cref{cor:susceptibility_correlation_length} implies that the two parts of the bound are of the same order when $\|x\|$ is of order $\xi(\beta)$.
We begin with the following elementary convolution estimate.

\begin{lemma}
\label{lem:threshold_convolution}
 Let $d\geq 1$ and $\alpha<1$. There exists a constant $C_1=C_1(d,\alpha)$ such that
\begin{multline*}\sum_{\substack{a\in \Z^d\\ \leq \|a\|\leq \|x\|/4}} \sum_{\substack{b\in \Z^d\\
\|b-x\|\leq \|x\|/4}}
\langle a\rangle^{-d-\alpha}\langle a-b\rangle^{-d-\alpha}\langle b-x\rangle^{-d-\alpha} \min\left\{1,\frac{\langle a\rangle}{R}\right\}\min\left\{1,\frac{\langle x-b\rangle}{R}\right\} \\\leq C_1 R^{-2\alpha} \langle x\rangle^{-d-\alpha}
\end{multline*}
for every $R\geq 1$ and $x\in \Z^d$ with $\|x\|\geq R$.
\end{lemma}

\begin{proof}
The $\langle a-b\rangle^{-d-\alpha}$ term is always of order $\langle x \rangle^{-d-\alpha}$. Thus, by symmetry, it suffices to prove that
\[\sum_{\substack{a\in \Z^d\\ \leq \|a\|\leq \|x\|/4}} 
\langle a\rangle^{-d-\alpha} \min\left\{1,\frac{\langle a\rangle}{R}\right\} \preceq R^{-\alpha}.\]
This can be done by summing over the possible distances of $a$ from $0$:
\[\sum_{\substack{a\in \Z^d\\ \leq \|a\|\leq \|x\|/4}} 
\langle a\rangle^{-d-\alpha} \min\left\{1,\frac{\langle a\rangle}{R}\right\} \preceq \frac{1}{R}\sum_{r=1}^R r^{d-1}r^{-d+1-\alpha}+\sum_{r=R}^{\|x\|} r^{d-1}r^{-d-\alpha} \preceq R^{-\alpha}\]
as claimed. (If $\alpha=1$ then an additional log factor appears, while if $\alpha>1$ we get a bound of $R^{-1}$ instead of $R^{-\alpha}$.)
\end{proof}


\begin{proof}[Proof of \cref{thm:subcritical}]
Since connection probabilities are monotone in $\beta$, it suffices by \cref{prop:main_pointwise_upper} to prove that
\[
\P_{\beta}(0\leftrightarrow x) \preceq 
 \chi(\beta)^2 \langle x\rangle^{-d-\alpha} \qquad \text{ for every $\|x\|\geq \xi(\beta)$.}
\] 
Given $n\geq 1$ and a long-range percolation configuration $\omega \in \{0,1\}^{\Z^d \times \Z^d}$, let $\omega_n$ be the configuration obtained by deleting all edges of Euclidean diameter at least $n$, and let $\P_{\beta,n}$ be the law of $\omega_n$ when $\omega$ has law $\P_\beta$. When $\beta<\beta_c$, the measure $\P_{\beta,n}$ is the law of a subcritical percolation configuration on a transitive, locally finite weighted graph, and it follows from the sharpness of the phase transition \cite{aizenman1987sharpness,MR852458,duminil2015new} that $\P_{\beta,n}(0\leftrightarrow x)$ decays exponentially in $\|x\|$ as $x\to \infty$ (with rate of decay depending on $\beta$ and $n$). As such, the constant
\[
A(\beta,n) = \inf\left\{A \geq 1: \P_{\beta,n}(0\leftrightarrow x) \leq 
A \chi(\beta)^2 \langle x\rangle^{-d-\alpha} \text{ for every $\|x\|\geq \xi(\beta)$} \right\}
\]
is finite for every $0< \beta<\beta_c$ and $n\geq 1$. We will prove that there exists a constant $C_1=C_1(d,\alpha,C,c)<\infty$ such that
\begin{equation}
\label{eq:Abetan_bootstrap}
A(\beta,n) \leq C_1 + \frac{1}{2} A(\beta,n)
\end{equation}
for every $0<\beta <\beta_c$ and $n\geq 1$; this will then imply that $A(\beta,n) \leq 2C_1$ for $0<\beta <\beta_c$ and $n\geq 1$, from which the claim will follow by continuity of measure. We fix $0\leq \beta<\beta_c$ and $n\geq 1$ and lighten notation by writing $\P=\P_{\beta,n}$. (We do not use the fact that we have deleted the long edges from the model other than to make $A(\beta,n)$ finite and allow safe rearrangement of the inequality \eqref{eq:Abetan_bootstrap}.)

Given $x\in \Z^d$, let $\mathscr{L}_x$ be the event that $0$ is connected to an edge by a simple open path that includes an edge of length (i.e., distance between endpoints) at least $\|x\|/16$. Summing over the possible locations $\{u,v\}$ of this edge and using the BK inequality, we obtain that
\begin{multline}
\P(\mathscr{L}_x) \preceq \beta \sum_{\|u-v\|\geq \|x\|/16} \langle u-v\rangle^{-d-\alpha} \P(0\leftrightarrow u)\P(v\leftrightarrow x) \\
\preceq \langle x\rangle^{-d-\alpha} \sum_{u,v\in \Z^d}\P(0\leftrightarrow u)\P(v\leftrightarrow x) = \chi(\beta)^2\langle x\rangle^{-d-\alpha}
\label{eq:Lx}
\end{multline}
for every $x\in \Z^d$ and $0<\beta \leq \beta_c$.

We now wish to bound the probability that $0$ is connected to $x$ by a simple path that does \emph{not} use any edges of length at least $\|x\|/16$.
Let $r$ be an integer parameter to be optimized over later in the proof, let $\Lambda$ be the box $\Lambda=[-r\xi(\beta),r\xi(\beta)]^d$, and let $x\in \Z^d$ satisfy $\|x\| \geq 32r \xi(\beta)$. For each $z\in \Z^d$, let $\Lambda^z=\Lambda+z$. If $\|z\| \leq r\xi(\beta)$ and $\|w-x\|\leq r\xi(\beta)$, then the boxes $\Lambda^z$ and $\Lambda^w$ are disjoint and contain $0$ and $x$ respectively. 
 On the event $\{0\leftrightarrow x \}\setminus \mathscr{L}_x$, the points $0$ and $x$ must be connected by some simple open path $\gamma$ that leaves $\Lambda^z$ for the first time via some edge $\{u,a\}$ with $\|u-a\|\leq \|x\|/8$ and enters $\Lambda^w$ for the first time via some edge $\{b,v\}$ with $\|b-v\|\leq \|x\|/8$. Summing over the possible locations of these two edges and applying the BK inequality yields that
\[
\P(\{0\leftrightarrow x\}\setminus \mathscr{L}_x) \preceq \sum_{u\in \Lambda^z} \sum_{v\in \Lambda^w} \sum_{\substack{a\in \Z^d \setminus \Lambda^z\\\|u-a\|\leq \|x\|/8}} \sum_{\substack{b\in \Z^d \setminus \Lambda^w \\\|b-v\|\leq \|x\|/8}}  \langle u-a\rangle^{-d-\alpha}
\langle v-b\rangle^{-d-\alpha}  \P(0 \leftrightarrow u) \P(v \leftrightarrow x) \P(a \leftrightarrow b)
\]
for every $z$ with $\|z\|\leq r \xi(\beta)$ and $w$ with $\|w\|\leq r\xi(\beta)$.
Averaging over $z$ and $w$, the same argument as in the proof of \cref{lem:random_box_convolution} yields that
\begin{multline*}
\P(\{0\leftrightarrow x \}\setminus \mathscr{L}_x) \preceq \sum_{\|u\|\leq 2r \xi(\beta)} \sum_{\|x-v\| \leq 2r\xi(\beta)} \sum_{\|u-a\| \leq \|x\|/4} \sum_{\|b-v\|\leq \|x\|/8}  \langle u-a\rangle^{-d-\alpha}
 \\
 \cdot \langle v-b\rangle^{-d-\alpha} \P(0 \leftrightarrow u) \P(v \leftrightarrow x) \P(a \leftrightarrow b) \min\left\{1,\frac{\|u-a\|}{r\xi(\beta)}\right\}\min\left\{1,\frac{\|v-b\|}{r\xi(\beta)}\right\}.
\end{multline*}
Since every term in the sum involves points $u$ and $v$ with $\|u-v\|\geq \frac{28}{32}\|x\|$ and hence $\|a-u\|\leq \frac{1}{4}\|x\| \leq \frac{1}{4}\|u-v\|$ and $\|b-v\|\leq \frac{1}{8}\|x\| \leq \frac{1}{4}\|u-v\|$, we can insert the estimate $\P(a \leftrightarrow b) \leq A\chi(\beta)^2 \langle a-b\rangle^{-d-\alpha}$ and apply \cref{lem:threshold_convolution} (to the sum over $a$ and $b$) to obtain that
\begin{multline*}
\P(\{0\leftrightarrow x\} \setminus \mathscr{L}_x) \preceq 
\frac{A\chi(\beta)^2}{r^{2\alpha} \xi(\beta)^{2\alpha}} \langle x \rangle^{-d-\alpha} \sum_{u,v\in \Z^d}\P(0\leftrightarrow u)\P(v \leftrightarrow x)
\\\preceq \frac{A\chi(\beta)^4}{r^{2\alpha} \xi(\beta)^{2\alpha}} \langle x \rangle^{-d-\alpha} \preceq \frac{A\chi(\beta)^2}{r^{2\alpha}} \langle x \rangle^{-d-\alpha}
\end{multline*}
for every $x$ with $\|x\|\geq 32 r\xi(\beta)$, where we used that $\chi(\beta)\asymp \xi(\beta)^\alpha$ in the last inequality. On the other hand, \cref{prop:main_pointwise_upper} implies that $\P(0\leftrightarrow x)\preceq \langle x \rangle^{-d+\alpha} \preceq r^{2\alpha}\chi(\beta)^2 \langle x \rangle^{-d-\alpha}$ for $\|x\|\leq 32 r \xi(\beta)$, and putting this together with \eqref{eq:Lx} we obtain that there exists a constant $C_1=C_1(d,\alpha,c,C)$ such that
\[
\P(0\leftrightarrow x) \leq C_1 \left(r^{2\alpha} + r^{-2\alpha}A\right) \chi(\beta)^2 \langle x \rangle^{-d-\alpha}
\] 
for every $0\leq \beta < \beta_c$ and $x\in \Z^d$. Taking $r=r(d,\alpha,c,C)$ sufficiently large that $C_1 r^{-2\alpha} \leq 1/2$ completes the proof.
\end{proof}

\subsection{Logarithmic divergence of the triangle for $d=3\alpha$}

In this section we apply \cref{prop:subcritical_upper} to prove \cref{cor:triangle}. The only thing that still requires a calculation is that the triangle diverges at most logarithmically when $\alpha=d/3<1$; the consequences regarding mean-field critical behaviour follow from the results of \cite{HutchcroftTriangle}.

\begin{proof}[Proof of \cref{cor:triangle}]
The proof will more generally allow us to bound the triangle diagram in terms of the susceptibility whenever $d\geq 1$ and $0<\alpha \leq d$ are such that the conclusions of \cref{thm:subcritical} hold; this should be the case for $\alpha$ less than the \emph{crossover value} $\alpha_c(d) \geq 1$.
We have by symmetry that
\begin{equation*}
\nabla_\beta 
\leq 3\sum_{x,y \in \Z^d} \P_{\beta}(0\leftrightarrow x)\P_{\beta}(x\leftrightarrow y)\P_{\beta}(y\leftrightarrow 0) \mathbbm{1}\left(\|x\|\leq \|y\|,\|x-y\|\right).
\end{equation*}
Note that if $\|y\|,\|x-y\|$ are both at least $\|x\|$ then they satisfy $2\|y\|\geq \|y\|+\|x-y\| - \|x\| \geq \|x-y\|$ and $2\|x-y\|\geq \|x-y\|+\|y\| - \|x\| \geq \|y\|$, so that $\frac{1}{2}\|x-y\| \leq \|y\|\leq 2\|x-y\|$. As such, considering the number of points with $\|y\|=r$ and using \cref{thm:subcritical} we obtain that if $\|x\|\leq \xi(\beta)$ then
\begin{multline*}
\sum_{y\in \Z^d} \P_{\beta}(x\leftrightarrow y)\P_{\beta}(y\leftrightarrow 0) \mathbbm{1}\left(\|x\|\leq \|y\|,\|x-y\|\right) 
\\\preceq \sum_{r=\|x\|}^{\lfloor \xi(\beta) \rfloor} r^{d-1} r^{-2d+2\alpha} + \chi(\beta)^2 \sum_{r=\lceil \xi(\beta) \rceil} r^{d-1}r^{-2d-2\alpha}
 \preceq \langle x\rangle^{-d+2\alpha} + \chi(\beta)^4 \xi(\beta)^{-d-2\alpha} \preceq \langle x\rangle^{-d+2\alpha},
\end{multline*}
where we used that $\chi(\beta)\asymp \xi(\beta)^\alpha$ in the final inequality, 
while if $\|x\|>\xi(\beta)$ then
\[
\sum_{y\in \Z^d} \P_{\beta}(x\leftrightarrow y)\P_{\beta}(y\leftrightarrow 0) \mathbbm{1}\left(\|x\|\leq \|y\|,\|x-y\|\right) \preceq \chi(\beta)^4\sum_{r=\|x\|}^\infty r^{d-1}r^{-2d-2\alpha} \preceq \chi(\beta)^4 \langle x\rangle^{-d-2\alpha}.
\]
Putting these two estimates together and using \cref{thm:subcritical} a second time we obtain that
\begin{align*}
\nabla_\beta 
&\preceq \sum_{r=1}^{\lfloor\xi(\beta)\rfloor} r^{d-1}r^{-d+\alpha} r^{-d+2\alpha} + 
\chi(\beta)^6 \sum_{r=\lceil\xi(\beta)\rceil}^{\infty} r^{d-1}r^{-d-\alpha} r^{-d-2\alpha} \\
&=
\sum_{r=1}^{\lfloor\xi(\beta)\rfloor} r^{-d-1+3\alpha} + 
\chi(\beta)^6 \sum_{r=\lceil\xi(\beta)\rceil}^{\infty} r^{-d-1-3\alpha}.
\end{align*}
The first term is $O(1)$ when $\alpha<d/3$, $O(\log \xi(\beta))$ when $\alpha=d/3$, and $O(\xi(\beta)^{3\alpha-d})$ when $\alpha>d/3$.
The second term is of order $\chi(\beta)^6 \xi(\beta)^{-d-3\alpha}$, which is $o(1)$ for $\alpha < d/3$, $O(1)$ for $\alpha = d/3$, and $O(\xi(\beta)^{3\alpha-d})$ for $\alpha>d/3$. Thus, we have that
\begin{equation}
\label{eq:triangle_correlation_length_bound}
\nabla_\beta \preceq \begin{cases}
1 & \alpha <d/3 \\
\log \xi(\beta) & \alpha= d/3\\
\xi(\beta)^{3\alpha-d} & \alpha > d/3.\\
\end{cases}
\end{equation}
It follows from the results of \cite{hutchcroft2020power,1901.10363} that $\chi(\beta)=|\beta-\beta_c|^{-O(1)}$ as $\beta\uparrow \beta_c$ whenever $\alpha <d$, and since $\chi(\beta)\asymp \xi(\beta)^\alpha$ when $\alpha<1$ we deduce that if $\alpha=d/3<1$ then $\nabla_\beta = O(\log 1/|\beta-\beta_c|)$ as $\beta\uparrow \beta_c$ as claimed.
\end{proof}

\begin{remark}
When $d\in \{1,2\}$ and $d/3<\alpha<2d/5$, the estimate \eqref{eq:triangle_correlation_length_bound} can be inserted into the differential inequality of \cite{HutchcroftTriangle} to obtain that
\[
\frac{d}{d\beta}\chi(\beta) \succeq \frac{\chi(\beta)(\chi(\beta)-\nabla_\beta)}{\nabla_\beta^2} \succeq \chi(\beta)^{2-6+\frac{2d}{\alpha}}
\]
in an appropriate interval $(\beta_c-\eps,\beta_c)$, 
and hence that
\[
\chi(\beta) =O\bigl( |\beta-\beta_c|^{-\alpha/(2d-5\alpha)} \bigr) 
\]
as $\beta \to\infty$. In particular, the exponent $\gamma$ satisfies $\gamma \leq \alpha/(2d-5\alpha)$ when it is well-defined and $\alpha<2d/5$. While this is presumably not a particularly good estimate, it does show that $\gamma \downarrow 1$ as $\alpha \downarrow d/3$. (It is an improvement on the bound $\gamma \leq (d+\alpha)/(d-\alpha)$, which follows from the bound $\delta\leq 2d/(d-\alpha)$ of \cite{hutchcroft2022sharp} together with the inequality $\gamma \leq \delta-1$ \cite{1901.10363}, only when $\alpha/d < \frac{1}{2}(\sqrt{3}-1) \approx 0.366$.)  Using the inequalities between critical exponents discussed in \cite{1901.10363,HutchcroftTriangle}, it follows that the exponents $\delta$ and $\beta$ also converge to their mean-field values of $2$ and $1$ as $\alpha \downarrow d/3$ when $d\in \{1,2\}$. It is conjectured that $\delta=(d+\alpha)/(d-\alpha)$ when $1\leq d <6$ and $d/3<\alpha<\alpha_c(d)$, but this has only been proven for the hierarchical version of the model \cite{hutchcroft2022sharp}. To our knowledge, no closed-form expressions are conjectured for $\gamma$ or $\beta$ when $d/3<\alpha<\alpha_c$.
\end{remark}

\subsection{Pointwise lower bounds}

It remains only to prove the lower bounds of \cref{thm:main_pointwise,thm:subcritical}. The proof of these lower bounds also gives a slightly stronger estimate that allows us to restrict our paths to a box of the same scale. Note that the constant $2$ can be replaced with any constant strictly greater than $1$ (with the constant $a$ depending on this choice of constant).

\begin{prop}
\label{prop:main_pointwise_lower}
Let $d\geq 1$, let $J:\Z^d\times \Z^d \to [0,\infty)$ be a symmetric, integrable, translation-invariant kernel, and suppose that there exist constants $0<\alpha<d$ and $c>0$ such that $c \|x-y\|^{-d-\alpha} \leq J(x,y) \leq C\|x-y\|^{-d-\alpha}$ for all distinct $x,y \in \Z^d$. If $\alpha<1$ then 
\[
\P_{\beta}(0\leftrightarrow x) \geq \P_{\beta}\bigl(0\leftrightarrow x \text{ in $[-2\|x\|,2\|x\|]^d$}\bigr)  \succeq \|x\|^{-d+\alpha} \left(1 \vee \frac{\|x\|}{\xi(\beta)}\right)^{-2\alpha}
\] 
for every $x\in \mathbb{Z}^d$ and $\beta_c/2 \leq \beta \leq \beta_c$. 
\end{prop}

The restriction  $\beta_c/2 \leq \beta \leq \beta_c$ allows us to safely ignore all factors of $\beta$ that might otherwise appear in our calculations.

\begin{proof}[Proof of \cref{prop:main_pointwise_lower}]
Fix $\beta_c/2\leq \beta\leq \beta_c$ and $x\in \Z^d$ with $\|x\|\geq 8$. Let $n = \lfloor \min\{\|x\|,\xi(\beta)\}/4\rfloor$, and let $\Lambda_1=\{y\in \Z^d:\|y\|\leq n\}$ and $\Lambda_2=\{y\in \Z^d:\|y-x\|\leq n\}$ be the boxes of side length $2n+1$ centred at $0$ and $x$ respectively. Consider the random variable $Z$, defined to be the number of pairs $(u,v)\in \Lambda_1 \times \Lambda_2$ such that $0$ is connected to $u$ inside $\Lambda_1$, $\{u,v\}$ is open, and $v$ is connected to $x$ inside $\Lambda_2$. If $Z>0$ then $0$ is connected to $x$ inside the box $[-2\|x\|,2\|x\|]^d$, so it suffices to prove that 
\[\P_{\beta}(Z>0) \succeq \|x\|^{-d+\alpha} \left(1 \vee \frac{\|x\|}{\xi(\beta)}\right)^{-2\alpha}.\] We will prove this by bounding the first and second moments of $Z$ and using the inequality $\P(X>0) \geq (\E X)^2/\E[X^2]$, which holds for all non-negative random variables as a consequence of Cauchy-Schwarz. For the first moment, we use indepence of the three conditions defining whether or not a pair $(u,v)$ contributes to $Z$ to get that
\begin{equation*}
\E_\beta [Z] \asymp \Biggl[\sum_{y\in [-n,n]^d}\P_{\beta}\Bigl(0\leftrightarrow y \text{ inside $[-n,n]^d$}\Bigr)\Biggr]^2 \|x\|^{-d-\alpha} 
\asymp n^{2\alpha} \|x\|^{-d-\alpha} 
\asymp \|x\|^{-d+\alpha} \left(1 \vee \frac{\|x\|}{\xi(\beta)}\right)^{-2\alpha} 
\end{equation*}
where the second estimate holds since $n \leq \xi(\beta)$. (More specifically, the upper bound follows from \eqref{eq:spatially_averaged} while the lower bound follows from \cref{lem:correlation_length}.)
 We now turn to the second moment. Let $K^1$ denote the cluster of the origin in $\Lambda_1$ and let $K^2$ denote the cluster of the origin in $\Lambda_2$. Since  $K^1$ and $K^2$ are independent of each other and of all edges with one endpoint in $\Lambda^1$ and the other in $\Lambda^2$, we have that
\[
\E_\beta\left[ \binom{Z}{2}\right] \leq \E_\beta\left[\binom{|K^1|}{2}\binom{|K^2|}{2}\right] \|x\|^{-2d-2\alpha} = \E_\beta\left[|K^1|^2\right]^2 \|x\|^{-2d-2\alpha}.
\]
Now, it follows from the \emph{universal tightness theorem} of \cite[Theorem 2.2]{hutchcroft2020power} (by a similar calculation to that of \cite[Lemma 2.9]{hutchcrofthierarchical}) that
\begin{equation*}
\E_\beta\left[|K^1|^2\right] \preceq \E_\beta |K^1| \cdot (\text{Median size of largest cluster in $\Lambda^1$ under $\P_{\beta_c}$})
\\
\preceq n^{\alpha} \cdot n^{(d+\alpha)/2},
\end{equation*}
where the upper bound $\E_\beta |K^1| \preceq n^\alpha$ follows from \eqref{eq:spatially_averaged} and the $n^{(d+\alpha)/2}$ upper bound on the median size of the largest cluster in $\Lambda^1$ follows from the proof of \cite[Proposition 2.2]{hutchcrofthierarchical} as noted in \cite[Remark 2.11]{hutchcrofthierarchical}. Putting these bounds together, we obtain that
\begin{multline*}
\E_\beta[Z^2]=\E_\beta[Z]+2\E_\beta\left[ \binom{Z}{2}\right] \\ \preceq \|x\|^{-d+\alpha} \left(1 \vee \frac{\|x\|}{\xi(\beta)}\right)^{-2\alpha} + n^{d+3\alpha}\cdot \|x\|^{-2d-2\alpha} \preceq \|x\|^{-d+\alpha} \left(1 \vee \frac{\|x\|}{\xi(\beta)}\right)^{-2\alpha}.
\end{multline*}
It follows that $\P_\beta(Z>0)\geq \E_\beta[Z]^2/\E_\beta[Z^2] \succeq \|x\|^{-d+\alpha} \left(1 \vee \frac{\|x\|}{\xi(\beta)}\right)^{-2\alpha}$ as claimed.
\end{proof}

\section*{Acknowledgements}
This work was initiated during a visit to Columbia University in March 2024 and completed mostly while the author was attending the Eighteenth Annual Workshop in Probability and Combinatorics at McGill University's Bellairs Institute in Holetown, Barbados. The work was supported by NSF grant DMS-1928930.

 \setstretch{1}
 \footnotesize{
  \bibliographystyle{abbrv}
  \bibliography{unimodularthesis.bib}

\begin{thebibliography}{10}

\bibitem{aizenman1987sharpness}
M.~Aizenman and D.~J. Barsky.
\newblock Sharpness of the phase transition in percolation models.
\newblock {\em Comm. Math. Phys.}, 108(3):489--526, 1987.

\bibitem{MR762034}
M.~Aizenman and C.~M. Newman.
\newblock Tree graph inequalities and critical behavior in percolation models.
\newblock {\em J. Statist. Phys.}, 36(1-2):107--143, 1984.

\bibitem{MR868738}
M.~Aizenman and C.~M. Newman.
\newblock Discontinuity of the percolation density in one-dimensional {$1/|x-
  y|^2$} percolation models.
\newblock {\em Comm. Math. Phys.}, 107(4):611--647, 1986.

\bibitem{MR4248721}
Y.~Aoun.
\newblock Sharp asymptotics of correlation functions in the subcritical
  long-range random-cluster and {P}otts models.
\newblock {\em Electron. Commun. Probab.}, 26:Paper No. 22, 9, 2021.

\bibitem{MR1127713}
D.~J. Barsky and M.~Aizenman.
\newblock Percolation critical exponents under the triangle condition.
\newblock {\em Ann. Probab.}, 19(4):1520--1536, 1991.

\bibitem{baumler2023distances}
J.~B{\"a}umler.
\newblock Distances in $1/\|x-y\|^{2d}$ percolation models for all dimensions.
\newblock {\em Communications in Mathematical Physics}, 404(3):1495--1570,
  2023.

\bibitem{baumler2022isoperimetric}
J.~B{\"a}umler and N.~Berger.
\newblock Isoperimetric lower bounds for critical exponents for long-range
  percolation.
\newblock {\em Annales de l'Institut Henri Poincare (B) Probabilites et
  statistiques}, 60(1):721--730, 2024.

\bibitem{behan2017scaling}
C.~Behan, L.~Rastelli, S.~Rychkov, and B.~Zan.
\newblock A scaling theory for the long-range to short-range crossover and an
  infrared duality.
\newblock {\em Journal of Physics A: Mathematical and Theoretical},
  50(35):354002, 2017.

\bibitem{MR1896880}
N.~Berger.
\newblock Transience, recurrence and critical behavior for long-range
  percolation.
\newblock {\em Comm. Math. Phys.}, 226(3):531--558, 2002.

\bibitem{biskup2021arithmetic}
M.~Biskup and A.~Krieger.
\newblock Arithmetic oscillations of the chemical distance in long-range
  percolation on $\mathbb{Z}^d$.
\newblock {\em arXiv preprint arXiv:2112.12365}, 2021.

\bibitem{brezin2014crossover}
E.~Brezin, G.~Parisi, and F.~Ricci-Tersenghi.
\newblock The crossover region between long-range and short-range interactions
  for the critical exponents.
\newblock {\em Journal of Statistical Physics}, 157(4-5):855--868, 2014.

\bibitem{MR782962}
D.~Brydges and T.~Spencer.
\newblock Self-avoiding walk in {$5$} or more dimensions.
\newblock {\em Comm. Math. Phys.}, 97(1-2):125--148, 1985.

\bibitem{chatterjee2020restricted}
S.~Chatterjee and J.~Hanson.
\newblock Restricted percolation critical exponents in high dimensions.
\newblock {\em Communications on Pure and Applied Mathematics},
  73(11):2370--2429, 2020.

\bibitem{MR3306002}
L.-C. Chen and A.~Sakai.
\newblock Critical two-point functions for long-range statistical-mechanical
  models in high dimensions.
\newblock {\em Ann. Probab.}, 43(2):639--681, 2015.

\bibitem{MR4032873}
L.-C. Chen and A.~Sakai.
\newblock Critical two-point function for long-range models with power-law
  couplings: the marginal case for {$d\ge d_{\rm c}$}.
\newblock {\em Comm. Math. Phys.}, 372(2):543--572, 2019.

\bibitem{ding2023uniqueness}
J.~Ding, Z.~Fan, and L.-J. Huang.
\newblock Uniqueness of the critical long-range percolation metrics.
\newblock {\em arXiv preprint arXiv:2308.00621}, 2023.

\bibitem{duminil2016absence}
H.~Duminil, V.~Sidoravicius, and V.~Tassion.
\newblock Absence of infinite cluster for critical {B}ernoulli percolation on
  slabs.
\newblock {\em Communications on Pure and Applied Mathematics},
  69(7):1397--1411, 2016.

\bibitem{duminil2015new}
H.~Duminil-Copin and V.~Tassion.
\newblock A new proof of the sharpness of the phase transition for {B}ernoulli
  percolation and the {I}sing model.
\newblock {\em Comm. Math. Phys.}, 343(2):725--745, 2016.

\bibitem{essam1978percolation}
I.~Essam, D.~Gaunt, and A.~Guttmann.
\newblock Percolation theory at the critical dimension.
\newblock {\em Journal of Physics A: Mathematical and General}, 11(10):1983,
  1978.

\bibitem{fitzner2015nearest}
R.~Fitzner and R.~van~der Hofstad.
\newblock Mean-field behavior for nearest-neighbor percolation in {$d>10$}.
\newblock {\em Electron. J. Probab.}, 22:Paper No. 43, 65, 2017.

\bibitem{gori2017one}
G.~Gori, M.~Michelangeli, N.~Defenu, and A.~Trombettoni.
\newblock One-dimensional long-range percolation: a numerical study.
\newblock {\em Physical Review E}, 96(1):012108, 2017.

\bibitem{grimmett2010percolation}
G.~Grimmett.
\newblock {\em Percolation}, volume 321 of {\em Grundlehren der Mathematischen
  Wissenschaften [Fundamental Principles of Mathematical Sciences]}.
\newblock Springer-Verlag, Berlin, second edition, 1999.

\bibitem{MR1043524}
T.~Hara and G.~Slade.
\newblock Mean-field critical behaviour for percolation in high dimensions.
\newblock {\em Comm. Math. Phys.}, 128(2):333--391, 1990.

\bibitem{MR1959796}
T.~Hara, R.~van~der Hofstad, and G.~Slade.
\newblock Critical two-point functions and the lace expansion for spread-out
  high-dimensional percolation and related models.
\newblock {\em Ann. Probab.}, 31(1):349--408, 2003.

\bibitem{heydenreich2015progress}
M.~Heydenreich and R.~van~der Hofstad.
\newblock {\em Progress in high-dimensional percolation and random graphs}.
\newblock CRM Short Courses. Springer, Cham; Centre de Recherches
  Math\'{e}matiques, Montreal, QC, 2017.

\bibitem{MR2430773}
M.~Heydenreich, R.~van~der Hofstad, and A.~Sakai.
\newblock Mean-field behavior for long- and finite range {I}sing model,
  percolation and self-avoiding walk.
\newblock {\em J. Stat. Phys.}, 132(6):1001--1049, 2008.

\bibitem{1901.10363}
T.~Hutchcroft.
\newblock New critical exponent inequalities for percolation and the random
  cluster model.
\newblock {\em Probab. Math. Phys.}, 1(1):147--165, 2020.

\bibitem{hutchcroft2020power}
T.~Hutchcroft.
\newblock Power-law bounds for critical long-range percolation below the
  upper-critical dimension.
\newblock {\em Probab. Theory Related Fields}, 181(1-3):533--570, 2021.

\bibitem{hutchcroft2022critical}
T.~Hutchcroft.
\newblock Critical cluster volumes in hierarchical percolation.
\newblock {\em arXiv preprint arXiv:2211.05686}, 2022.

\bibitem{HutchcroftTriangle}
T.~Hutchcroft.
\newblock On the derivation of mean-field percolation critical exponents from
  the triangle condition.
\newblock {\em Journal of Statistical Physics}, 189(1):6, 2022.

\bibitem{hutchcroft2022sharp}
T.~Hutchcroft.
\newblock Sharp hierarchical upper bounds on the critical two-point function
  for long-range percolation on zd.
\newblock {\em Journal of Mathematical Physics}, 63(11), 2022.

\bibitem{hutchcrofthierarchical}
T.~Hutchcroft.
\newblock The critical two-point function for long-range percolation on the
  hierarchical lattice.
\newblock {\em The Annals of Applied Probability}, 34(1B):986--1002, 2024.

\bibitem{hutchcroft2023high}
T.~Hutchcroft, E.~Michta, and G.~Slade.
\newblock High-dimensional near-critical percolation and the torus plateau.
\newblock {\em The Annals of Probability}, 51(2):580--625, 2023.

\bibitem{kesten1980critical}
H.~Kesten.
\newblock The critical probability of bond percolation on the square lattice
  equals $1/2$.
\newblock {\em Communications in mathematical physics}, 74(1):41--59, 1980.

\bibitem{MR879034}
H.~Kesten.
\newblock Scaling relations for {$2$}{D}-percolation.
\newblock {\em Comm. Math. Phys.}, 109(1):109--156, 1987.

\bibitem{MR2551766}
G.~Kozma and A.~Nachmias.
\newblock The {A}lexander-{O}rbach conjecture holds in high dimensions.
\newblock {\em Invent. Math.}, 178(3):635--654, 2009.

\bibitem{MR2748397}
G.~Kozma and A.~Nachmias.
\newblock Arm exponents in high dimensional percolation.
\newblock {\em J. Amer. Math. Soc.}, 24(2):375--409, 2011.

\bibitem{lawler2002one}
G.~Lawler, O.~Schramm, W.~Werner, et~al.
\newblock One-arm exponent for critical 2d percolation.
\newblock {\em Electronic Journal of Probability}, 7, 2002.

\bibitem{MR3723429}
M.~Lohmann, G.~Slade, and B.~C. Wallace.
\newblock Critical two-point function for long-range {$O(n)$} models below the
  upper critical dimension.
\newblock {\em J. Stat. Phys.}, 169(6):1132--1161, 2017.

\bibitem{luijten1997interaction}
E.~Luijten.
\newblock {\em Interaction range, universality and the upper critical
  dimension}.
\newblock PhD thesis, Technische Universiteit Delft, Mekelweg 5, 2628 CD Delft,
  Netherlands, 1997.

\bibitem{MR852458}
M.~V. Menshikov.
\newblock Coincidence of critical points in percolation problems.
\newblock {\em Dokl. Akad. Nauk SSSR}, 288(6):1308--1311, 1986.

\bibitem{newman1986one}
C.~M. Newman and L.~S. Schulman.
\newblock One dimensional $1/|j- i|^s$ percolation models: The existence of a
  transition for $s \leq 2$.
\newblock {\em Communications in Mathematical Physics}, 104(4):547--571, 1986.

\bibitem{nienhuis1987coulomb}
B.~Nienhuis.
\newblock Coulomb gas formulation of two-dimensional phase transitions.
\newblock {\em Phase transitions and critical phenomena}, 11:1--53, 1987.

\bibitem{picco2012critical}
M.~Picco.
\newblock Critical behavior of the ising model with long range interactions.
\newblock {\em arXiv preprint arXiv:1207.1018}, 2012.

\bibitem{sak1973recursion}
J.~Sak.
\newblock Recursion relations and fixed points for ferromagnets with long-range
  interactions.
\newblock {\em Physical Review B}, 8(1):281, 1973.

\bibitem{sakai2018crossover}
A.~Sakai.
\newblock Crossover phenomena in the critical behavior for long-range models
  with power-law couplings (stochastic analysis on large scale interacting
  systems).
\newblock {\em Mathematical Analysis Research Institute Kokyuroku Special
  Edition}, 79:51--62, 2020.

\bibitem{schulman1983long}
L.~S. Schulman.
\newblock Long range percolation in one dimension.
\newblock {\em Journal of Physics A: Mathematical and General}, 16(17):L639,
  1983.

\bibitem{MR2239599}
G.~Slade.
\newblock {\em The lace expansion and its applications}, volume 1879 of {\em
  Lecture Notes in Mathematics}.
\newblock Springer-Verlag, Berlin, 2006.
\newblock Lectures from the 34th Summer School on Probability Theory held in
  Saint-Flour, July 6--24, 2004, Edited and with a foreword by Jean Picard.

\bibitem{MR3772040}
G.~Slade.
\newblock Critical exponents for long-range {$O(n)$} models below the upper
  critical dimension.
\newblock {\em Comm. Math. Phys.}, 358(1):343--436, 2018.

\bibitem{smirnov2001critical2}
S.~Smirnov.
\newblock Critical percolation in the plane: conformal invariance, {C}ardy's
  formula, scaling limits.
\newblock {\em Comptes Rendus de l'Acad{\'e}mie des Sciences-Series
  I-Mathematics}, 333(3):239--244, 2001.

\bibitem{smirnov2001critical}
S.~Smirnov and W.~Werner.
\newblock Critical exponents for two-dimensional percolation.
\newblock {\em Math. Res. Lett.}, 8(5-6):729--744, 2001.

\bibitem{van1982instability}
A.~C. van Enter.
\newblock Instability of phase diagrams for a class of" irrelevant"
  perturbations.
\newblock {\em Physical Review B}, 26(3):1336, 1982.

\end{thebibliography}
  }

\end{document}